\documentclass[12pt]{article}
\usepackage{titlesec}
\usepackage[utf8]{inputenc}
\usepackage{hyperref}
\usepackage{verbatim}
\usepackage{amssymb}
\usepackage{amsmath}
\usepackage{dsfont}
\usepackage{amsthm}
\usepackage{tikz}
\usepackage{tikz-cd}
\usepackage{enumerate}
\usepackage{cleveref}

\usepackage{setspace}

\usetikzlibrary{shapes.geometric,positioning,arrows}

\theoremstyle{theorem}
\newtheorem{theorem}{Theorem}
\newtheorem{lemma}[theorem]{Lemma}
\newtheorem{corollary}[theorem]{Corollary}

\usepackage{mathtools}
\DeclarePairedDelimiterX{\bra}[2]{\langle}{\rangle}{#1, #2}
\DeclarePairedDelimiterX{\Bra}[1]{\langle}{\rangle}{#1}

\newcommand{\norm}[1]{{\Vert #1 \Vert }}

\DeclareMathOperator{\dist}{\mathrm{dist}}
\newcommand{\dv}{\,\mathrm{d}V}
\newcommand{\da}{\,\mathrm{d}A}

\begin{document}

\title{Short-time heat content asymptotics via the wave and eikonal equations}
\author{Nathanael Schilling \thanks{ Zentrum Mathematik, Technische Universität München, Boltzmannstr. 3, 85748 Garching bei München, schillna@ma.tum.de}}
\maketitle
\begin{abstract}
In this short paper, we derive an alternative proof for some known \cite{Vandenberg2015} short-time asymptotics of the heat content in a compact full-dimensional submanifolds $S$ with smooth boundary. This includes formulae like 
\begin{equation*}
\int_{S} \exp(t\Delta)\left( f \mathds 1_S\right) \dv = \int_S f  \dv -  \sqrt{\frac{t}{\pi}} \int_{\partial S} f \da + o(\sqrt t),\quad t \rightarrow 0^+\,,
\end{equation*}
and explicit expressions for similar expansions involving other powers of $\sqrt t$.
By the same method, we also obtain short-time asymptotics of $\int_S \exp(t^m\Delta^m) \left(f \mathds 1_S\right)\dv$,  $m \in \mathbb N$, and more generally for one-parameter families of operators $t \mapsto k(\sqrt{-t\Delta})$ defined by an even Schwartz function $k$.
\end{abstract}
\section{Introduction}
Let $(M,g)$ be a complete, boundaryless\footnote{we assume that $M$ has no boundary for the sake of simplicity, the method presented here can be adapted to more general manifolds with boundary provided that $S$ is compactly contained in the interior of $M$.
If this is not the case, such as in the classical heat content setting as in \cite{Vandenberg1994},
it should be possible to obtain similar results by modifying the geometrical optics construction used.
}, oriented Riemannian manifold with Laplace--Bel\-tra\-mi operator $\Delta$,
and volume $\dv$. On a codimension-$1$ submanifold of $M$, we write $\da$ for the induced surface (hyper)-area form.
The \emph{heat semigroup} $T_t \coloneqq \exp(t\Delta)$ acting on $L^2(M,\dv)$ is well-defined ($\Delta$ is essentially self-adjoint on $C^\infty_c(M)$ \cite{Chernoff1973}) and its behaviour as $t \rightarrow 0^+$ has been extensively investigated in the literature.
Specifically, for a set $S \subset M$, the \emph{heat content} of the form $\Omega_{S,f}(t) \coloneqq \int_S T_t(f \mathds 1_S)\dv$, $f \in C^\infty(M)$,  has recently
received much attention; see, for instance, \cite{Miranda2007,Vandenberg2015,Vandenberg2018} and the references therein.

Let us briefly recall some known results.
On $\mathbb R^n$, sets $S$ of \emph{finite perimeter} $P(S)$ are characterized by \cite[Thm.~3.3 ]{Miranda2007}
\begin{equation}\label{eq:res1}
\lim\limits_{t \rightarrow 0^+ }\sqrt{\frac{\pi}{t}}\Big(\Omega_{S,\mathds 1_M}(0) - \Omega_{S,\mathds 1_M}(t)\Big) = P(S)\,.
\end{equation}
Extensions of this idea to abstract metric spaces are given in \cite{Marola2016}.
In the setting of compact manifolds $M$ (or $M = \mathbb R^n$) and $S$ a full-dimensional submanifold with smooth boundary $\partial S$, the authors of \cite{Vandenberg2015} show that

\begin{equation}\label{eq:res2}
\Omega_{S,f}(t) = \sum_{j=0}^\infty \beta_j t^\frac{j}{2},\quad t\rightarrow 0^+,
\end{equation}
where the coefficients $\beta_j$ depend on $S$, $f$ and the geometry of $M$. The setting of \cite{Vandenberg2015} is more general, amongst other things it includes $f$ which have singularities. Some of the cofficients obtained in \cite[corollary 1.7]{Vandenberg2015} are
\begin{equation*}
\beta_0 = \int_S f \dv\,,\qquad \beta_1 = -\frac{1}{\sqrt \pi}\int_{\partial S} f \da
\,,\qquad\beta_2 = \frac{1}{2}\int_{S} \Delta  f \dv\,.
\end{equation*}
Extensions to some non-compact manifolds $M$ and certain non-compact $S$ are in \cite{Vandenberg2018}.

Both \cref{eq:res1,eq:res2} are proven with significant technical effort, yielding strong results. For example,  
in \cite{Miranda2007}, explicit knowledge of the fundamental solution of the heat equation is used
to obtain \cref{eq:res1} for $C^{1,1}$-smooth $\partial S$, after which geometric measure theory is used. Similarly, 
\cite{Vandenberg2015} requires pseudo-differential calculus and invariance theory.

Our aim is to show that slightly weaker
results can be obtained by considerably lower technical effort.
In contrast to \cite{Miranda2007}, we treat only compact $S$ with smooth boundary,
and do not allow $f$ to have singularities like \cite{Vandenberg2015} does.
On the other hand, we put no further restrictions than completeness on $M$.
The proof presented here is simple, comparatively short, and provides an alternative differential geometric/functional analytic point of view to questions regarding heat content.
Moreover, this approach is readily extended to some other PDEs including the semi-group generated by $\Delta^m$. 
Observe that $T(t) = k(\sqrt{-t\Delta})$ with $k(x) = \exp(-x^2)$. We allow $k$ to be an arbitrary even Schwarz function, with $\Omega_{S,f}(t) = \int_{S} k(\sqrt{-t\Delta})(f \mathds 1_S)\dv$ and will prove:
\begin{theorem}\label{thm:thm1}
Let $M$ be a complete Riemannian manifold with Laplace-Beltrami operator $\Delta$, Riemannian volume $\dv$ and induced (hyper) area form $\da$.
Let $S \subset M$ be a compact full-dimensional submanifold with smooth boundary. For $f \in C^\infty(M)$ and $N \in \mathbb N$,
\begin{equation*}
\Omega_{S,f}(t) = \sum_{j=0}^N \beta_j t^{\frac{j}{2}} + o(t^\frac{N}{2})\,,\quad t \rightarrow 0^+\,,
\end{equation*}
for constants $(\beta_j)_{j=0}^N$ described further in the next theorem.
\end{theorem}
With the $j$-th derivative $k^{(j)}$  (for $j \in \mathbb N_0$), let $r_{j} \coloneqq (-1)^{j/2} k^{(j)}(0)$ for $j$ even and $r_{j} \coloneqq (-1)^{(j-1)/2} \int_0^\infty \frac{2k^{j}(s)}{-\pi s}\, \mathrm d s$ for $j$ odd. Let $\varphi$ locally be the signed distance function (see also \cite[section 3.2.2]{Petersen2016}) to $\partial S$ with $S = \varphi^{-1}([0,\infty))$, and denote by $\nabla$ and $\cdot$ the gradient and (metric) inner product respectively. The vector field $\nu \coloneqq -\nabla \varphi$ is outer unit normal at $\partial S$.
\begin{theorem}\label{thm:thm2}
The coefficients of \cref{thm:thm1} satisfy $\beta_0 = r_0 \int_S f \dv$ and $\beta_1 = -\frac12 r_1 \int_{\partial S} f \da$. For even $j \in \mathbb N_{\geq 2}$, 
\begin{equation*}
\beta_{j} = \frac{r_j}{j!}
\int_S\frac12 {\Delta^{j/2} f}\dv
\end{equation*}
Moreover, given the Lie-derivative $\mathcal L_\nu$ with respect to $\nu$,
\begin{align*}
\beta_3 = \frac{r_3}{2\cdot3!} \int_{\partial S}\mathcal L_\nu (-\mathcal L_\nu + \frac12\Delta\varphi)  f - \frac12 \Delta f + \frac12 (-\mathcal L_\nu + \frac12\Delta\varphi)^2 f\da\,,
\end{align*}
similar expression can be found also for larger odd values of $j$ (see \cref{sec:higher_coeffs}).
\end{theorem}
The properties of the signed distance function $\varphi$ may be used to express terms appearing in \Cref{thm:thm2} using other quantities.
For example, its Hessian $\nabla^2 \varphi$ is the second fundamental form on the tangent space of $\partial S$ \cite[ch.~3]{Gray2004},
and thus $\frac12\Delta\varphi$ is the mean curvature.

Our approach to prove \cref{thm:thm1,thm:thm2} is to combine 3 well-known facts:
\begin{enumerate}[(A)]
\item The short-time behaviour of the heat flow is related to the short-time behaviour of the wave equation (cf.~\cite{Cheeger1982}).
\item The short-time behaviour of the wave equation with discontinous initial data is related to the short-time behaviour of the eikonal equation (cf.~`geometrical optics' and the progressing wave expansion \cite{Taylor2011}).
\item The short-time behaviour of the wave and eikonal equations with initial data $f \mathds 1_S$ is directly related to the geometry of $M$ near $\partial S$.
\end{enumerate}
Though points (A)-(C) are well known in the literature, they have (to the best of our knowledge) not been applied
to the study of heat content so far.

A significant portion of
(C) will rest on an application of the Reynolds transport theorem. Here, 
denote by $\Phi^s$ the time-$s$ flow of the vector field $\nu = - \nabla \varphi$.
For small $s$, the (half) tubular neighborhood
\begin{equation}\label{eq:nghb}
S^{-s} \coloneqq \{x \in M \setminus S : \mathrm{dist}(x,\partial S) \leq s\}
\end{equation}
satisfies $S \cup S^{-s} = \Phi^s(S)$.
For $a \in C^\infty((-\varepsilon, \varepsilon) \times M)$, by \cite[Ch.~V, Prop.~5.2]{Lang1995}, 
\begin{align}\nonumber
\frac{\mathrm d}{\mathrm ds}\left.\int_{S^{-s}} a(s,\cdot)\dv\right\vert_{s=0}
&=\frac{\mathrm d}{\mathrm ds}\left(\left.\int_{S^{-s} \cup S} a(s,\cdot)\dv- \int_{S} a(s,\cdot)\dv\right)\right\vert_{s=0}  \\
&=\int_{S} \mathcal L_{\tilde \nu} [ a(0,\cdot) \dv]= \int_{\partial S} a(0,\cdot)\da\,.\label{eq:mainformula}
\end{align}
The last equation is a consequence of Cartan's magic formula and Stokes' theorem,
where we use that $\dv(\nu, \cdot) = \da(\cdot)$ on $\partial S$.

\section{Proof for $\beta_0,\beta_1$}
By Fourier theory (for non-Gaussian $k$, the formulae must be adapted),
\begin{equation*}\label{eq:ft}
k(t) = \exp(- t^2) = \int_0^\infty \hat k(s) \cos(ts) \,\mathrm{d}s \quad
\mathrm{with}\quad \hat k(s) \coloneqq \frac{1}{\sqrt{\pi}}\exp\left(\frac{-s^2}{4 }\right)\,.
\end{equation*}
On the operator level, this yields the well-known formula \cite[section~6.2]{Taylor2011}
\begin{equation}\label{eq:transmutation}
T_t = \exp(t\Delta) = \int_0^\infty \hat k(s) \cos( s \sqrt{-t\Delta}) \,\mathrm ds\,.
\end{equation}
The operator $W^s \coloneqq \cos(s\sqrt{-\Delta})$ is the time-$s$ solution operator for the wave equation with zero initial velocity,
in particular $u(s,x) \coloneqq (W^s f \mathds 1_S)(x)$ (weakly) satisfies $(\partial_t^2 - \Delta)u = 0$.
Let $\Bra{\cdot,\cdot}$ denote the $L^2(M,\dv)$ inner product. Using \cref{eq:transmutation},
\begin{equation*}\label{eq:transmutation2}
\Bra{T_t f \mathds 1_S, \mathds 1_S} = \int_0^\infty \hat k(s)\Bra{W_{s\sqrt t}f\mathds 1_S, \mathds 1_{S}}\, \mathrm{d} s\,.
\end{equation*}
Similar reasoning has been used to great effect in 
\cite{Cheeger1982} to derive heat-kernel bounds by making use of the \emph{finite propagation speed} of the wave equation.
As in \cite{Cheeger1982}, finite propagation speed yields for $s \geq 0$ that
$\Bra{W_{s} f \mathds 1_S, \mathds 1_{M \setminus S}} = \Bra{W_{s} f \mathds 1_{S^{s}}, \mathds 1_{S^{-s}}}$, where $S^{s} \coloneqq (M \setminus S)^{-s}$ is defined like \cref{eq:nghb}. Even if $\mathds 1_{M \setminus S} \notin L^2(M,\dv)$, we have just seen that the inner product $\Bra{W_s f \mathds 1_S, \mathds 1_{M \setminus S}}$ is nevertheless well-defined. 
In \cite{Cheeger1982}, it is further observed that $\norm{W_s} \leq 1$. Using the Cauchy-Schwarz inequality and assuming $f = \mathds 1_M$,  \cref{eq:mainformula} yields
\begin{equation}\label{eq:csi}
h(s) \coloneqq \Bra{W_{s} f \mathds 1_{S^{s}}, \mathds 1_{S^{-s}}} \leq
\norm{\mathds 1_{S^s}}_2 \norm{\mathds 1_{S^{-s}}}_2 
\leq s \int_{\partial S} \da + o(s), \quad s \rightarrow 0^+.
\end{equation}
In addition, $|\Bra{W_s f \mathds 1_S, \mathds 1_{S}}| \leq \norm{f\mathds 1_S}_2\norm{\mathds 1_S}_2$ for all $s\geq 0$, in particular as $s \rightarrow \infty$.
We conclude with some calculations (cf.~\cref{lemma:heatkernellemma} below), that
\begin{align}\nonumber
\Bra{T_t \mathds 1_S, \mathds 1_{S}} &= \int_0^\infty\hat k(s) \left( \Bra{W_{s\sqrt t } \mathds 1_S, \mathds 1_M} - \Bra{W_{s \sqrt t} \mathds 1_S, \mathds 1_{M\setminus S}}\right)\,\mathrm{d} s\\
\label{eq:crude}
&= \Bra{\mathds 1_S, \mathds 1_M} - \int_0^\infty \hat k(s) h(s\sqrt t)\, \,\mathrm d s\\
& \geq \int_S \dv - 2\sqrt{\frac{t}{\pi}} \int_{\partial S} \da + o(\sqrt t), \qquad t \rightarrow 0^+.\nonumber
\end{align}

This is weaker than the desired estimate, and restricts to $f = \mathds 1_M$. The problem is that the estimates in \cref{eq:csi} are too crude.
To improve them, we instead approximate the solution $u$ to the wave equation with geometrical optics, using the ``progressing wave'' construction described in \cite[section 6.6]{Taylor2011}, some details of which we recall here.
The basic idea is that $u$ is in general discontinuous, with an outward-- and an inward-- moving discontinuity
given by the zero level-set of functions $\varphi^+$ and $\varphi^-$ respectively.
The functions $\varphi^\pm$ satisfy the eikonal equation
$\partial_t \varphi = \pm|\nabla \varphi^\pm|$ with inital value $\varphi^\pm(0,\cdot) = \varphi(\cdot)$.
Equivalently, using the the (nonlinear) operator $Ew \coloneqq (\partial_t w )^2 - |\nabla w|^2$, the 
functions $\varphi^\pm$ satisfy $E(\varphi^\pm)=0$.
Our analysis is greatly simplified by choosing the initial $\varphi$ to (locally) be the signed distance function to $\partial S$.
The eikonal equation is then $\partial_t \varphi^\pm = \pm|\nabla \varphi| = \pm|-\nu|= \pm1$, i.e.~ $\varphi^\pm(x,t) = \varphi(x) \pm t$. 

The progressing wave construction further makes use of two (locally existing and smooth) solutions $a^\pm_0$ to the first-order
transport equations $\pm\partial_t a^\pm_0(t,\cdot) + \nu \cdot \nabla a^\pm_0(t,x) = \frac12 a^\pm_0\Delta\varphi^\pm$.
Observe that with the Heaviside function
$\theta \colon \mathbb R \rightarrow \mathbb R$, and 
$\Box \coloneqq \partial_t^2 - \Delta$,
the expression $\Box(a_0^\pm \theta(\varphi^\pm))$ is given by
\begin{align}
(\theta''(\varphi^\pm)E\varphi^\pm + \Box\varphi^\pm\theta'(\varphi^\pm))a_0^\pm +
2\left(\partial_t a_0^\pm\partial_t \varphi^\pm - \nabla a_0^\pm \cdot \nabla \varphi^\pm\right)\theta'(\varphi^\pm) + \Box a_0^\pm \theta(\varphi^\pm).\nonumber
\end{align}
The functions $\varphi^\pm$ and $a_0^\pm$ have been chosen so the above simplifies to
\begin{align}
\label{eq:progressing_wave_general}
\Box(a_0^\pm\theta(\varphi^\pm))&=  2\left(\pm \partial_t a_0^\pm + \nabla a_0^\pm \cdot \nu -\frac12\Delta \varphi a^\pm_0 \right)\theta'(\varphi^\pm) + \Box a_0^\pm \theta(\varphi^\pm)\nonumber \\ &= \Box a_0^\pm \theta(\varphi^\pm) \,.
\end{align} 
Thus $\Box(a_0^\pm \theta(\varphi^\pm))$ is as smooth as $\theta$ is. We use
\begin{equation*}\label{eq:geometricoptics}
\tilde u(t,x) \coloneqq a^+_0(t,x)\theta(\varphi^+(t,x)) + a^-_0(t,x)\theta(\varphi^-(t,x))
\end{equation*}
as an approximation to the discontinuity of the solution $u$ to the wave-equation.
To maintain consistency with the initial values of $u$,
the initial values of the approximation $\tilde u$ are chosen to coincide with those of $u$ at $t=0$,
this is achieved by setting $a_0^\pm(0,\cdot) = \frac 12 f$ so that (at least formally) $\partial_t \tilde u(0,\cdot) = 0$ and also $\tilde u(0,\cdot) = \mathds 1_S f$.

The function $\tilde u$ approximates the discontinuous solution $u$ of the wave-equation well enough that the function $(s,x) \mapsto u(s,x) - \tilde u(s,x)$ is 
continuous on $[-T,T] \times M$, see \cite[section 6.6, eq.~6.35]{Taylor2011}. By construction, $\tilde u(0,\cdot) = u(0,\cdot)$. Hence
$|(u(s,x) - \tilde u (s,x)| = o(1)$ as $s \rightarrow 0^+$,
which implies
\begin{equation}\label{eq:part1}
|\Bra{u(s,\cdot), \mathds 1_{S^{-s}}} - \Bra{\tilde u(s,\cdot),\mathds 1_{S^{-s}}}| = o(s)\,\quad s \rightarrow 0^+\,.
\end{equation}
As $\nabla \varphi = -\nu$, for sufficiently small $t$ the sets $\{x \in M : \varphi^+(t,x) = 0\}$ (resp. $\{x : \varphi^-(t,x) = 0\}$) are level sets of $\varphi$ on the outside (resp. inside) of $S$ (see also \cite[section 6.6]{Taylor2011}).
By construction, $\theta(\varphi^-)$ vanishes outside of $S$ for $t > 0$.
Consequently, using \cref{eq:mainformula}, we see that as $s \rightarrow 0^+$,
\begin{align}
\Bra{\tilde u (s,\cdot),\mathds 1_{S^{-s}}} &=  \int_{S^{-s}} a^+_0(s,x) \mathds 1_{ \{\varphi^+(s,\cdot) \geq 0 \}}  + a^-_0(s,x) \mathds 1_{\{\varphi^-(s,x) \geq 0 \}}  \dv(x)\nonumber\\
&=  s\int_{\partial S} a_0^+(0,x) \da(x) + o(s)
=  \frac{s}{2}\,\int_{\partial S} f \da + o(s).\label{eq:part2}
\end{align}
Combining \cref{eq:part1,eq:part2},
\begin{equation*}\label{eq:fromprevious}
h(s) = \Bra{W_s f \mathds 1_S, \mathds 1_{S^{-s}}} = \Bra{u(s,\cdot),\mathds 1_{S^{-s}}} = \frac{s}{2} \int_{\partial S} f \da + o(s),\quad s \rightarrow 0^+.
\end{equation*}
Calculations along the lines of \cref{lemma:heatkernellemma,eq:crude} yield
\begin{equation}\nonumber
\Bra{T_t f \mathds 1_S, \mathds 1_S} =\int_S f \dv -  \sqrt{\frac{t}{\pi}} \int_{\partial S}f \da + o(\sqrt t),\qquad t \rightarrow 0^+,
\end{equation}
as claimed.
\begin{lemma}\label{lemma:heatkernellemma}
Let $j \in \mathbb N$ and $\gamma: \mathbb R_{\geq 0} \rightarrow \mathbb R$. Let $\gamma(s) = s^j + o(s^j)$ for $s \rightarrow 0$ and $\gamma(s) = O(1)$ for $s \rightarrow \infty$. Then for $t \rightarrow 0^+$,
\begin{equation}\label{eq:generalft}
\int_0^\infty \gamma(s \sqrt t)\hat k(s) \,\mathrm ds = t^\frac{j}{2}\begin{cases} 
(-1)^{\frac{j}{2}}\,k^{(j)}(0)  & \textrm{$j$ even } \\
(-1)^{\frac{j-1}{2}}\int_0^\infty \frac{2\, k^{(j)}(s)}{-\pi s}\, \mathrm d s & \textrm{$j$ odd}
\end{cases} \quad + o\left(t^\frac{j}{2}\right)\,.
\end{equation}
With $k(s) = \exp(-s^2)$ and $h(s) = c_0 + c_1s + c_2 s^2 + o(s^2)$, this implies
\begin{equation}\label{eq:specialcase}
\int_0^\infty h(s\sqrt t) \hat k(s) \,\mathrm d s = c_0 + \frac{2c_1}{\sqrt \pi} \sqrt t + 2 c_2 t + o(t)\,.
\end{equation}
\end{lemma}
\begin{proof}
For even $j$, we obtain \cref{eq:generalft} by the Fourier-transform formula for$j$-th derivatives. If $j$ is odd, 
we also need to multiply by the sign-function in frequency space, and then use that the inverse Fourier-transform (unnormalized) of the sign function is given by the principal value $\mathrm{p.v.}\left(\frac{2i}{x}\right)$ \cite[section 4]{Taylor2011}, see also \cite[Chapter 7]{Rudin1991}. 
\Cref{eq:generalft} holds more generally,
e.g. if $k$ is an even Schwarz function. \Cref{eq:specialcase} may also be verified directly without \cref{eq:generalft}. 
\end{proof}
\section{Proof for $\beta_2,\beta_3,\cdots$}\label{sec:higher_coeffs}
We now turn to calculating $\beta_j$ for $j \geq 2$.
We use the $N$-th order progressing wave construction with sufficiently large $N \gg j$.
For the sake of simplicity, we write $O(t^\infty)$ for quantities that can be made $O(t^k)$ for any $k\in\mathbb N$ by choosing sufficiently large $N$.
As in the previous section, the construction is from \cite[section 6.6]{Taylor2011}.
With $\theta_0 \coloneqq \theta$, and $\theta_i(t) \coloneqq \int_{-\infty}^t \theta_{i-1}(s)\mathrm ds$ we write
\begin{align*}
\tilde u^\pm(t,x) \coloneqq \sum_{i=0}^N a^\pm_i(t,x) \theta_i(\varphi^\pm(t,x))\,.
\end{align*}
Here the functions $a_0^\pm$ are defined as before; and for $i \geq 1$ the $i$-th order transport equations $\pm \partial_t a^\pm_i = -\nu \cdot \nabla a^\pm_i + \frac 12 a_i^\pm \Delta \varphi^\pm - \frac 12\Box a_{i-1}^\pm$  define $a^\pm_i$ together
with initial data $a^\pm_i(0,\cdot) = -\frac12(\partial_t a^+_{i-1}(0,\cdot) + \partial_t a^-_{i-1}(0,\cdot))$. As in
\cref{eq:progressing_wave_general}, one may verify that $\Box\tilde u^\pm  = \Box a_{i} \theta_N(\varphi^\pm)$.
Writing $\tilde u = \tilde u^+ + \tilde u^-$ and 
\begin{equation}\nonumber
u(t,x) = \tilde u^+(t,x) + \tilde u^-(t,x) + R_N(t,x)\,,
\end{equation}
the remainder satisfies $R_N \in C^{(N,1)}([-T,T] \times M)$ and $R_N(t,\cdot)$ vanishes at $t=0$, see \cite[section 6.6, eq.~6.35]{Taylor2011}. Moreoever, $R_N$ is supported on $\{(x,t) : \dist(x,S) \leq |t|\}$, all of this implies that, as $t \rightarrow 0^+$,
\begin{align}\label{eq:approximationworks}
h(t) = \int_{M \setminus S} u(t,x)\dv(x) = \int_{M \setminus S} \tilde u^+(t,x)\dv(x) + O(t^\infty)\,
\end{align}
and moreover $h \in C^\infty([0,T])$.
The structure of $R_N$ implies that $\Box \tilde u^+(t,x) = O(t^\infty)$ on $M \setminus S$,
provided that this expression is interpreted in a sufficiently weak sense.
Formally, therefore
\begin{align}
\partial_t^2\int_{M\setminus S} \tilde u^+(\cdot,t)\dv &= \int_{M \setminus S} \Delta \tilde u^+(\cdot,t) \dv + O(t^\infty) \nonumber \\
&= - \int_{\partial S} \nabla \tilde u^+(\cdot,t) \cdot \nu \da + O(t^\infty)\,,\label{eq:afterdivergence}
\end{align}
where the last step is the divergence theorem. One may verify \cref{eq:afterdivergence} rigorously by either doing the above steps in the sense
of distributions, or by a (somewhat tedious) manual computation.  
Combining this with \cref{eq:approximationworks}, 
\begin{align}\label{eq:hdd}
h''(t) = -\int_{\partial S} \nabla \tilde u^+(\cdot,t) \cdot \nu \da + O(t^\infty)\,.
\end{align}
The quantity $h^{(j)}(0)$ may thus be seen to depend $\tilde u^+(0,\cdot)$ at $\partial S$, which in turn depends on $a_i^\pm$ at $t=0$. 
Defining $\mathbf S_i \coloneqq a_i^+ + a_i^-$ and $\mathbf D_i \coloneqq a_i^+ - a_i^-$ for $i=0,1,\dots$,
let $L$ be the (spatial) differential operator defined for $w \in C^\infty(M)$ by $Lw \coloneqq \frac12 \Delta \varphi w - \nu \cdot \nabla w$.
For $i \in \mathbb N_0$, the transport equations imply
\begin{alignat}{5}\label{eq:recurrence_relations_a}
\partial_t \mathbf S_0 &= L\mathbf D_0\,,  &\quad\quad \partial_t \mathbf D_0 &= L \mathbf S_0\,,  \quad\quad\\
\partial_t \mathbf S_{i+1} &= L\mathbf D_{i+1} - \frac12\Box \mathbf D_{i}\,, &\quad\quad \partial_t \mathbf D_{i+1} &= L\mathbf S_{i+1} - \frac12\Box S_{i} \quad \mathrm{for} \quad i \geq 0\,,\label{eq:recurrence_relations_b}
\end{alignat}
with initial values satisfying
\begin{alignat}{4}\label{eq:recurrence_relations_iv_a}
a_{0}^+(0,\cdot) &=\  \frac12 \mathbf S_0(0,\cdot) = \frac12 f(\cdot)\,,\quad\quad &\mathbf D_0(0,\cdot) &= 0\,,\\
a_{i+1}^+(0,\cdot) &= \frac12 \mathbf D_{i+1}(0,\cdot) = -\frac12 \partial_t \mathbf S_{i}(0,\cdot)\,,\quad\quad &\mathbf S_{i+1}(0,\cdot) &= 0\label{eq:recurrence_relations_iv_b}\,.
\end{alignat}
\begin{lemma}\label{lemma:ailemma}
For $i,n \in \mathbb N_0$ it holds that $\partial_t^{2n}\mathbf D_i(0,\cdot) = 0$  (note that as a consequence, also $a_{i+1}(0,\cdot)$, $L\mathbf D_i(0,\cdot)$, and $\Box^n \mathbf D_i(0,\cdot)$ are zero).
\end{lemma}
\begin{proof}
We will proceed by induction over $i$ and use the identities \crefrange{eq:recurrence_relations_a}{eq:recurrence_relations_iv_b}. For $i=0$, $\mathbf D_0(0,\cdot) = 0$ is trivially satisfied. Moroever, $\partial_t^{2n} \mathbf D_0 = R^n\mathbf D_0$, which is zero at $t=0$. For $i=1$, observe that $a_1^+(0,\cdot) = -\frac12 \partial_t \mathbf S_0(0,\cdot) = -\frac12 L\mathbf D_0(0,\cdot) = 0$, and thus $\mathbf D_1(0,\cdot) = 0$. Likewise, $\partial_t^2 \mathbf D_1 = \partial_t(L\mathbf S_1 - \frac12\Box \mathbf S_0) = L(L\mathbf D_1 - \frac 12\Box \mathbf D_0) - \frac12\Box L\mathbf D_0$. As the operator $L$ commutes with $\partial_t^2$, this expression vanishes at $t=0$. Induction over $n$ proves the remainder of of the statement for $i=1$.
For the general case, we assume the induction hypothesis for $i$ and $i+1$ and start by noting that $\mathbf D_{i+2}(0,\cdot) = 2a_{i+2}^+(0,\cdot) = -\partial_t \mathbf S_{i+1}(0,\cdot) = -\left(L \mathbf D_{i+1}(0,\cdot) - \frac12\Box \mathbf D_{i}(0,\cdot) \right) = 0$. Moreover, $\partial_t^2 \mathbf D_{i+2} = \partial_t ( L \mathbf S_{i+2} - \frac12 \Box \mathbf S_{i+1}) = L(L\mathbf D_{i+2} - \frac12 \Box \mathbf D_{i+1}) - \frac12 \Box \left(L \mathbf D_{i+1} - \frac12 \Box \mathbf D_i\right)$, which again vanishes at $t=0$; the case $n > 1$ may again be proven by induction over $n$.
\end{proof}
\begin{corollary}\label{lemma:coefficientslemma} For even $j \in \mathbb N_{\geq 2}$, the $j$-th derivative of $h$ sastifies
\begin{equation}\nonumber
h^{(j)}(0) =  -\frac 12 \int_S \Delta^{j/2} f\dv\,.
\end{equation}
\end{corollary}
\begin{proof}

\Cref{lemma:ailemma} shows that for $i \geq 1$, $a^+_i(0,x) = 0$. Together with \cref{eq:hdd}, thus $h''(0) = -\int_{\partial S} \nabla a^+_0(0,\cdot) \cdot \nu \da = -\frac12\int_{\partial S} \nabla f \cdot \nu\da$. This is the case $j=2$. 
More generally, 
for $j = 2k$ with $k \in \mathbb N_{\geq 2}$, we use that (for $x \in \partial S$), $\tilde u^+$ satisfies
$\partial_t^2 \tilde u^+(t,x) = \Delta \tilde u^+(t,x) + O(t^\infty)$. \Cref{eq:hdd} ensures that as $t \rightarrow 0^+$,
\begin{align*}
 h^{(2k)}(t) = \int_{\partial S} \nabla (\Delta^{k-1} \tilde u^+(t,\cdot)) \cdot \nu \da + O(t^\infty)\,.
\end{align*}
As for the case $k=1$, it follows that $h^{(2k)}(0) = -\int_{\partial S} \nabla(\Delta^{k-1} a^+_0) \cdot \nu\da$, the divergence theorem yields the claim.
\end{proof}
The odd coefficients are trickier, we only compute the case $j=3$. We start with the observation that for $x\in\partial S$, $\varphi^+(t,x) = t$ and therefore
\begin{align*}\label{eq:vanishinghigher}
\tilde u^+(t,x) &= \sum_{i=0}^N \frac{1}{i!} t^i a^+_i(t,x)\quad \mathrm{ for }\ \ t \geq 0,\ \  x \in \partial S\,.
\end{align*}
Recall that that the Lie-derivative acts on functions $w \in C^\infty(M)$ by $\mathcal L_\nu w = \nabla w \cdot \nu$. Thus
$\mathcal L_\nu \theta_{i+1}(\varphi^+(t,x)) = -\theta_{i}(\varphi^+(t,x))$, so for $x \in \partial S$,
\begin{align*}
\mathcal L_\nu \tilde u^+(t,x) &= \sum_{i=0}^{N-1} \frac{t^i}{i!}(\mathcal L_\nu a_i^+(t,x) -  a_{i+1}(t,x)) + O(t^\infty)\,.
\end{align*}
Therefore $\partial_t \mathcal L_\nu \tilde u^+(0,x) = \partial_t(\mathcal L_\nu a_0^+(0,x) - a_1^+(t,x)) + (\mathcal L_\nu a_1^+(0,x) - a_2^+(0,x))$,
but the second term is zero as $a_1^+$ and $a_2^+$ vanish at $t=0$ by \cref{lemma:ailemma}. Substituting the transport equations and removing further zero terms leaves
$\partial_t \mathcal L_\nu \tilde u^+(0,x) = \mathcal L_\nu La_0^+(0,x) + \frac12 \Box a_0(0,x) = \frac12\left(\mathcal L_\nu L f(x) - \frac12 \Delta f(x) + \frac12 L^2 f(x) \right)$. Thus (recall that $L = - \mathcal L_\nu + \frac12\Delta \varphi$) directly from \cref{eq:hdd},
\begin{align*}
h^{(3)}(0) &= -\frac12\int_{\partial S}\mathcal L_\nu L f(x) - \frac12 \Delta f(x) + \frac12 L^2 f(x)\da(x)\,.
\end{align*}
The formula
\begin{equation}\label{eq:basicidea}
\Omega_{S,f}(t) = \int_0^\infty \hat k(s)\left(\int_S f \dv -  h( s \sqrt t ) \right) \,\mathrm{d} s
\end{equation}
established in the previous section, together with \cref{lemma:heatkernellemma}, yields the asymptotic behaviour of $\Omega_{S,f}(t)$ by
taking the Taylor-expansion of $h$ using \cref{lemma:coefficientslemma}. This gives the remainder of the claims of \cref{thm:thm2}.
\section{Discussion}
The above-said is not specific to the heat equation. Taking $k(x) = \exp(-x^{2m})$, $m \in \mathbb N$, we may, for example, study the one-parameter operator family $\exp(-t^m \Delta^m)$. The wave equation estimates needed are the same.  For $m \geq 2$, a brief calculation yields the explicit $t \rightarrow 0^+$ asymptotics 
\begin{equation*}
\Bra{\exp(t^m\Delta^m)f \mathds 1_S, \mathds 1_S} = \int_S f\dv -  \left(\pi^{-1} \Gamma\left(\frac{2m-1}{2m}\right) \int_{\partial S} f\da\right)\sqrt t  + o(t).
\end{equation*}
We conclude with the observation that the generalization of this paper to \emph{weighted} Riemannian manifolds (cf.~\cite{Grigoryan2009}) is straightforward.
\section{Acknowledgements}
The author was supported by the Priority Programme
SPP 1881 Turbulent Superstructures of the Deutsche Forschungsgemeinschaft.
The author thanks the reviewer for simplifying a significant part of the argument,
and thanks Oliver Junge and Daniel Karrasch for helping to improve the manuscript.


\begin{thebibliography}{}

\bibitem{Cheeger1982}
J.~Cheeger, M.~Gromov, and M.~Taylor.
\newblock Finite propagation speed, kernel estimates for functions of the
  {L}aplace operator, and the geometry of complete {R}iemannian manifolds.
\newblock {\em Journal of Differential Geometry}, 17(1):15--53, 1982.

\bibitem{Chernoff1973}
P.R.~Chernoff.
\newblock Essential self-adjointness of powers of generators of hyperbolic
  equations.
\newblock {\em Journal of Functional Analysis}, 12(4):401--414, 1973.

\bibitem{Gray2004}
A~Gray.
\newblock {\em Tubes}, volume 221 in Progress in Mathematics.
\newblock Springer, Basel, 2 edition, 2004.

\bibitem{Grigoryan2009}
A.~Grigor'yan.
\newblock {\em Heat Kernel and Analysis on Manifolds}.
\newblock Number~47 in Studies in Advanced Mathematics. AMS, Rhode Island, 2009.

\bibitem{Lang1995}
S.~Lang.
\newblock {\em Differential and Riemannian Manifolds}, volume 160 of {\em
  Graduate Texts in Mathematics}.
\newblock Springer New York, 1995.

\bibitem{Marola2016}
N.~Marola, M.~Miranda, and N.~Shanmugalingam.
\newblock Characterizations of sets of finite perimeter using heat kernels in
  metric spaces.
\newblock {\em Potential Analysis}, 45(4):609--633, 2016.

\bibitem{Miranda2007}
Jr~Miranda, M., D.~Pallara, F.~Paronetto, and M.~Preunkert.
\newblock {Short-time Heat Flow and Functions of Bounded Variation in $R^N$}.
\newblock {\em Annales de la Facult{\'e} des sciences de Toulouse:
  Math{\'e}matiques}, Ser. 6, 16(1):125--145, 2007.

\bibitem{Petersen2016}
P.~Petersen.
\newblock {\em Riemannian Geometry}, volume 171 of {\em Graduate Texts in Mathematics}
\newblock Springer, New York, 3 edition, 2016.

\bibitem{Rudin1991}
W.~Rudin.
\newblock {\em Functional analysis. International series in pure and applied
  mathematics}.
\newblock McGraw-Hill, Inc., New York, 1991.

\bibitem{Taylor2011}
M.~E. Taylor.
\newblock {\em Partial Differential Equations I}, volume 115 of {\em Applied
  Mathematical Sciences}.
\newblock Springer New York, 2 edition, 2011.

\bibitem{Vandenberg2018}
M.~van~den Berg.
\newblock Heat content in non-compact {R}iemannian manifolds.
\newblock {\em Integral Equations and Operator Theory}, 90(1):8, 2018.

\bibitem{Vandenberg2015}
M.~van~den Berg and P.~Gilkey.
\newblock Heat flow out of a compact manifold.
\newblock {\em The Journal of Geometric Analysis}, 25:1576--1601, 2015.

\bibitem{Vandenberg1994}
M.~Vandenberg and P.~Gilkey.
\newblock Heat content asymptotics of a {R}iemannian manifold with boundary.
\newblock {\em Journal of Functional Analysis}, 120(1):48--71, 1994.

\end{thebibliography}

\end{document}